\documentclass[a4paper,11pt]{article}
\usepackage[utf8]{inputenc}
\usepackage{color}
\usepackage{microtype}
\usepackage{enumerate}
\usepackage{amsmath}
\usepackage{amssymb}
\usepackage{fullpage}
\usepackage{array}
\usepackage{hyperref}
\usepackage{multirow}

\newtheorem{lemma}{Lemma}[section]
\newtheorem{theorem}[lemma]{Theorem}

\newtheorem{conjecture}[lemma]{Conjecture}

\newcounter{claim}

\newenvironment{proof}[1][]%
 {\noindent {\setcounter{claim}{0}\sc proof ---
   }{#1}{}}{\hfill$\Box$\vspace{2ex}} 

\newenvironment{claim}[1][]%
{\refstepcounter{claim}\vspace{1ex}\noindent{(\it\arabic{claim}){#1}{}}\it}{\vspace{1ex}}

\newenvironment{proofclaim}[1][]%
	{\noindent {}{#1}{}}{ This proves~(\arabic{claim}).\vspace{1ex}}

\usepackage{tikz}
\tikzstyle{vertex}=[circle, draw, inner sep=0pt, minimum size=6pt]
\newcommand{\vertex}{\node[vertex]}

\newcommand{\sm}{\setminus} %prive de 
\newcommand{\ov}{\overline}
\newcommand{\mc}{\mathcal}

\newenvironment{myp}[1]{{\noindent \it \bf \sf Proof\/ #1:}}{
\hfill {{\vrule height5pt width5pt depth0pt}\hskip 0cm \bigskip}
}

\newcommand{\ol}[1]{\overline{#1}}

\title{A new class of graphs that satisfies the Chen-Chv\'atal Conjecture \footnote{Partially supported by Basal program PBF 03 and Núcleo Milenio Información y Coordinación en Redes ICM/FIC P10-024F.   }}
\author{P. Aboulker$^a$,  M. Matamala $^{b,c}$, P. Rochet$^{c,d}$ and J. Zamora$^e$ \\
\small ($a$) Project Coati, I3S (CNRS, UNSA) and INRIA, Sophia Antipolis, France \\
\small ($b$) Depto. Ingeniería Matemática (DIM), Universidad de Chile \\
\small ($c$) Centro de Modelamiento Matemático (CMM, UMI 2807 CNRS), Universidad de Chile\\
\small ($d$) Laboratoire de Mathématiques Jean Leray, Université de Nantes \\
\small ($e$) Depto. Matemáticas, Universidad Andres Bello \\}
\begin{document}

\maketitle

\begin{abstract}
A well-known combinatorial theorem says that a set of $n$ non-collinear
points in the plane determines at least $n$ distinct lines.
Chen and Chvátal conjectured that this theorem extends  to 
metric spaces, with an appropriated definition of line.
In this work we prove a slightly stronger version of Chen and Chvátal conjecture
for a family of graphs containing chordal graphs and distance-hereditary graphs.
%As a by product, we prove that a minimal counterexample to 
%Chen and Chvátal conjecture is a prime 2-connected, non-chordal graph which
%gives a common generalization to previously known results.
\end{abstract}

\section{Introduction}
A classic result in Euclidean geometry asserts  that every non-collinear set of $n$ points in the Euclidean plane determines at least $n$ distinct lines.

Erd\H{os} \cite{E43} showed that this result is a consequence of the Sylvester-Gallai theorem which asserts that every non-collinear set of $n$ points in the plane determines a line containing precisely two points. 
%Later, De Bruijn and Erd\H{o}s~\cite{dbe} proved a more general combinatorial result which implies Theorem~\ref{dbe}.
Coxeter \cite{Cox} showed that the Sylvester-Gallai theorem holds in a more basic setting known as  \textit{ordered geometry}. Here, the notions of distance and angle are not used and, instead, a ternary relation of \textit{betweenness} is employed.
We write $[abc]$ for the statement that $b$ lies between $a$ and $c$.  
In this notation, a \emph{line} $\ov{xy}$ is defined (for any two distinct points $x$ and $y$) as:

\begin{equation} \label{definitionofline}
\ov{xy}=\{x,y\} \cup \{ u:[uxy]\text{ or } [xuy] \text{ or } [xyu]\}
\end{equation}

Betweenness in metric spaces was first studied by Menger \cite{menger} and further on  by Chv\'atal~\cite{ChvatalMetric}. 
In a metric space $(V,d)$, we define
\[
[abc]  \Leftrightarrow d(a,b) + d(b,c) = d(a,c).
\]

Hence, in any metric space $(V,d)$, we can define the line $\ov{uv}$ induced by two points $u$ and $v$ as in (\ref{definitionofline}). A line of a metric space $(V,d)$ is \textit{universal} if it contains all points of $V$. With this definition of lines in metric spaces, Chen and Chv\'{a}tal \cite{CC} proposed the following beautiful conjecture.

\begin{conjecture} \label{conjC}
 Every metric space on $n$ points, where $n\ge 2$, either has at least $n$ distinct lines or has a universal line.
\end{conjecture}

The best known lower bound for the number of lines in metric spaces with  no universal line is $\Omega(\sqrt n)$~\cite{metricSpace}.

As it is explained in~\cite{AK}, it suffices to prove Conjecture~\ref{conjC} for metric spaces with integral distances. 
This  motivates looking at two particular types of metric spaces. 
First, for a positive integer $k$, we define a \textit{k-metric space} to be a metric space in which all distances are integral and are at most $k$.
Chv\'{a}tal \cite{Chvatal2} proved that every $2$-metric space on $n$ points $(n \ge 2)$ either has at least $n$ distinct lines or has a universal line. The question is open for $k \ge 3$. Aboulker et al.~\cite{metricSpace} proved that, for all $k \ge 3$, a $k$-metric space with no universal line has at least $n/5k$ distinct lines. 

A second type of metric space with integral distances arises from graphs.
Any finite connected graph induces a metric space on its vertex set, where the distance between two vertices $u$ and $v$ is defined as the length of a shortest path linking $u$ and $v$. Such  metric spaces are called  \emph{graph metrics} and are the subject of this paper. 
The best known lower bound on the number of lines in a  graph metric with no universal line is  $\Omega(n^{4/7})$~\cite{metricSpace}.  
In~\cite{BBCCCCFZ} and~\cite{AK} it is proved that Conjecture~\ref{conjC} holds for chordal graphs and for distance-hereditary graphs respectively. 
The main result of this paper is to prove Conjecture~\ref{conjC} for all  graphs that can be constructed from chordal graphs by repeated substitutions and gluing along vertices. This  generalizes chordal and distance hereditary graphs.  %The exact definition of our class is given in the next section. 
%In particular, we prove that if a minimal counter example of Conjecture~\ref{conjC} exists, then it must be a prime $2$-connected non-chordal graph. 

\section{Statement of the main theorem}

Let $G=(V,E)$ be a connected  graph.
Let $a,b,c$ be three distinct vertices in $V(G)$. 
The distance $d_G(a,b)$ (or simply $d(a,b)$ when the context is clear) between $a$ and $b$ is the length of a shortest path linking $a$ and $b$.
We write $[abc]^G$ (or simply $[abc]$) when $d(a,b)+d(b,c)=d(a,c) < \infty$.
%When $G$ is not connected, $[abc]^G$ implies that $a,b$ and $c$ lie in the same connected component. 
Observe that $[abc] \Leftrightarrow [cba]$.  
We denote by $\ov{ab}^G$ (or simply $\ov{ab}$) the line induced by two distinct vertices $a,b$.
Recall that $\ov{ab}^G=\{a,b\} \cup \{x: [abx]\text{ or } [axb] \text{ or } [xab]\}$.
Notice that with this definition the line defined by two 
vertices $a,b$ lying in different connected components is $\{a,b\}$.
 We denote by ${\cal L}(G)$ the set of distinct lines in $G$ 
and by ${\ell }(G)=|{\cal L}(G)|$ the number of distinct lines in $G$. 

We denote by $N_G(v)$ the set of all neighbors of a vertex $v$ in $G$.
For a set of vertices $S$, we denote by $N_G(S)$ (or simply $N(S)$) the set of
all vertices \emph{outside} $S$ having a neighbor in $S$ .
A set $S$ is \emph{dominating} if $S\cup N_G(S)=V(G)$. 

A set of vertices $M$ of a graph $G=(V,E)$ is a 
\emph{module} if for each $a,b\in M$, $u\notin M$, $au\in E$ if and only if $bu\in E$. 
It is a \emph{non-trivial} module if $|V| > |M| \geq 2$.
If $M$ is a dominating set, we call it a \emph{dominating module}. 
In this situation, $N(M)$ is also a module unless $M=V$. 
When $M=\{u,v\}$, we say that  $(u,v)$ is  \emph{a pair of twins}.
If $u$ and $v$ are adjacent they  are called \emph{true} twins; otherwise, they are called \emph{false} twins.
A graph without non-trivial module is called a \emph{prime} graph.

A \emph{bridge} $ab$  is an edge whose deletion disconnects the graph. 
We denote by \emph{$br(G)$} the number of bridges of $G$. 
If $br(G)=0$, we say that $G$ is \emph{bridgeless}. 
If $ab$ is a bridge of a graph $G$, then for every vertex $p \in V(G) \sm \{a,b\}$, we either have $[pab]$ or $[abp]$. 
Hence $\ov{ab}^G=V(G)$ and thus Conjecture~\ref{conjC}  is only interesting for bridgeless graphs.

Let $\cal C$ be the class of graphs $G$ such that every induced subgraph of $G$ is either a chordal graph, has a cut-vertex or a non-trivial module. 
By definition, this class is \emph{hereditary}, that is, if $G\in \cal C$, then every induced subgraph of $G$ is also in $\mathcal C$. 

%For  instance, the graphs $H_5$, $\overline C_6$ and $H_8$, previously described, do not belong to $\cal C$. They are not chordal as they all contain $C_4$ as induced subgraph,  they are 2-connected and do not have non-trivial module.

Let ${\cal F}=\{C_4,K_{2,3},W_4,W'_4,K'_6,K'_8\}$ (see Figure~\ref{fig}).
In this work we prove the following theorem.
\begin{theorem}\label{th:main}
For each connected graph $G\in {\cal C}\setminus \cal F$, $\ell(G)+br(G)\geq |G|$.
\end{theorem}

\begin{figure}[h]
\centering
\begin{tikzpicture}[scale=0.6]
\begin{scope}[xshift=0cm]
\vertex[circle,  minimum size=8pt](1) at  (0,0) {};
\vertex[circle,  minimum size=8pt](2) at  (0,2) {};
\vertex[circle,  minimum size=8pt](3) at  (2,2) {};
\vertex[circle,  minimum size=8pt](4) at  (2,0) {};
\draw (1)--(2) -- (3) -- (4)--(1);
\draw (1,-0.7) node{$C_4$}; 
\end{scope}

\begin{scope}[xshift=4cm]
\vertex[circle,  minimum size=8pt](1) at  (0,0) {};
\vertex[circle,  minimum size=8pt](2) at  (0,2) {};
\vertex[circle,  minimum size=8pt](3) at  (2,2) {};
\vertex[circle,  minimum size=8pt](4) at  (2,0) {};
\vertex[circle,  minimum size=8pt](5) at  (1,1) {};
\draw (1)--(2) -- (3) -- (4)--(1);
\draw (1)--(5) -- (3);
\draw (1,-0.7) node{$K_{2,3}$}; 
\end{scope}

\begin{scope}[xshift=8cm]
\vertex[circle,  minimum size=8pt](1) at  (0,0) {};
\vertex[circle,  minimum size=8pt](2) at  (0,2) {};
\vertex[circle,  minimum size=8pt](3) at  (2,2) {};
\vertex[circle,  minimum size=8pt](4) at  (2,0) {};
\vertex[circle,  minimum size=8pt](5) at  (1,1) {};
\draw (1)--(2) -- (3) -- (4)--(1);
\draw (1)--(5) -- (3);
\draw (2)--(5) ;
\draw (1,-0.7) node{$W'_4$}; 
\end{scope}

\begin{scope}[xshift=12cm]
\vertex[circle,  minimum size=8pt](1) at  (0,0) {};
\vertex[circle,  minimum size=8pt](2) at  (0,2) {};
\vertex[circle,  minimum size=8pt](3) at  (2,2) {};
\vertex[circle,  minimum size=8pt](4) at  (2,0) {};
\vertex[circle,  minimum size=8pt](5) at  (1,1) {};
\draw (1)--(2) -- (3) -- (4)--(1);
\draw (1)--(5) -- (3);
\draw (2)--(5) -- (4);
\draw (1,-0.7) node{$W_4$}; 
\end{scope}

\begin{scope}[xshift=0.5cm, yshift=-4.8cm]
\vertex[circle,  minimum size=8pt](1) at  (0,0) {};
\vertex[circle,  minimum size=8pt](2) at  (2,0) {};
\vertex[circle,  minimum size=8pt](3) at  (4,0) {};
\vertex[circle,  minimum size=8pt](4) at  (0,2) {};
\vertex[circle,  minimum size=8pt](5) at  (2,2) {};
\vertex[circle,  minimum size=8pt](6) at  (4,2) {};
\draw (1)--(5) -- (3) -- (4)--(2)--(6)--(1) -- (2) -- (3);
\draw (4)--(5)--(6);
\draw(1)  to[bend right=20]   (3);  
\draw (4)  to[bend left=20]   (6);  
\draw (2,-1.1) node{$K_6'$}; 
\end{scope}

\begin{scope}[xshift=7cm, yshift=-4.8cm]
\vertex[circle,  minimum size=8pt](1) at  (0,0) {};
\vertex[circle,  minimum size=8pt](2) at  (2,0) {};
\vertex[circle,  minimum size=8pt](3) at  (4,0) {};
\vertex[circle,  minimum size=8pt](4) at  (6,0) {};
\vertex[circle,  minimum size=8pt](5) at  (0,2) {};
\vertex[circle,  minimum size=8pt](6) at  (2,2) {};
\vertex[circle,  minimum size=8pt](7) at  (4,2) {};
\vertex[circle,  minimum size=8pt](8) at  (6,2) {};
\draw (1)--(6)--(3) --(8) --(2) --(5) --(4) --(7) --(1) --(8)  ;
\draw (5)--(3);
\draw (4)--(6);
\draw (2)--(7);
\draw (1)--(2)--(3)--(4);
\draw (5)--(6)--(7)--(8);
\draw(1)  to[bend right=18]   (3);  
\draw (5)  to[bend left=18]   (7);  
\draw(2)  to[bend right=18]   (4);  
\draw (6)  to[bend left=18]   (8);  
\draw(1)  to[bend right=22]   (4);  
\draw (5)  to[bend left=22]   (8);
\draw (3,-1.4) node{$K_8'$}; 
\end{scope}
\end{tikzpicture}
  \caption{Graphs in  ${\cal F}$.}
  \label{fig}
\end{figure}
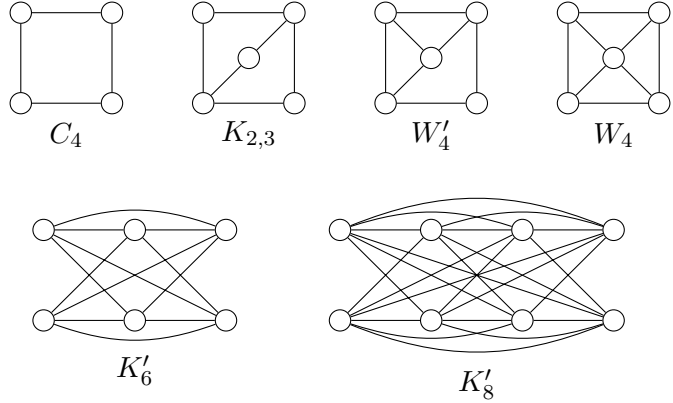

As a consequence we have that Chen-Chvatal conjecture holds for 
 $\cal C$, as it holds for graphs in  $\cal F$ (they all have a universal line). 
Since all distance-hereditary graphs contain either a pendant edge or a pair of twins, $\mathcal C$ is a super class of distance hereditary graphs. It is also clearly a super class of chordal graphs. \\

The difference between our result and the original conjecture is that we count a universal line as any other line but, since each bridge defines a universal line, we count it with multiplicity. It is tempting to conjecture that the property $ \ell(G)+br(G)\geq |G|$ holds for all graphs but a finite number. We know this to be false, as it was pointed out to us by Yori Zwols, owing to a simple observation: a counter-example containing a bridge produces an infinite number of counter-examples from replacing a bridge by a path of arbitrary length. So far, the three known minimal counter-examples containing a bridge are shown in Figure \ref{fig:ceb}. It remains unknown however whether all counter-examples to $\ell(G)+br(G)\geq |G|$ can be obtained from a finite set of graphs upon replacing a bridge by a path. \\

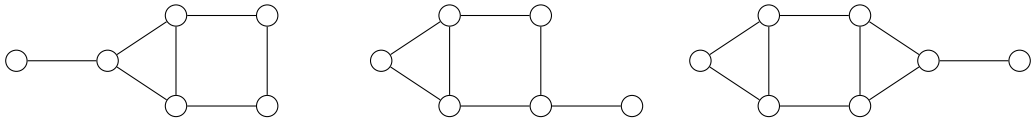
\begin{figure}[h]\label{fig:ceb}
\centering
\begin{tikzpicture}[scale=0.6]
\begin{scope}[xshift=0cm]
\vertex[circle,  minimum size=8pt](1) at  (0,0) {};
\vertex[circle,  minimum size=8pt](2) at  (0,2) {};
\vertex[circle,  minimum size=8pt](3) at  (2,2) {};
\vertex[circle,  minimum size=8pt](4) at  (2,0) {};
\vertex[circle,  minimum size=8pt](5) at  (-1.5,1) {};
\vertex[circle,  minimum size=8pt](6) at  (-3.5,1) {};
\draw (1)--(2) -- (3) -- (4)--(1);
\draw (1)--(5) -- (2);
\draw (5) -- (6);
\draw (1,-1) ;
%node{$H_5'$}; 
\end{scope}

\begin{scope}[xshift=6cm]
\vertex[circle,  minimum size=8pt](1) at  (0,0) {};
\vertex[circle,  minimum size=8pt](2) at  (0,2) {};
\vertex[circle,  minimum size=8pt](3) at  (2,2) {};
\vertex[circle,  minimum size=8pt](4) at  (2,0) {};
\vertex[circle,  minimum size=8pt](5) at  (-1.5,1) {};
\vertex[circle,  minimum size=8pt](6) at  (4,0) {};
\draw (1)--(2) -- (3) -- (4)--(1);
\draw (1)--(5) -- (2);
\draw (4) -- (6);
\draw (1,-1) ;
%node{$H_5'$}; 
\end{scope}

\begin{scope}[xshift=13cm]
\vertex[circle,  minimum size=8pt](1) at  (0,0) {};
\vertex[circle,  minimum size=8pt](2) at  (0,2) {};
\vertex[circle,  minimum size=8pt](3) at  (2,2) {};
\vertex[circle,  minimum size=8pt](4) at  (2,0) {};
\vertex[circle,  minimum size=8pt](5) at  (-1.5,1) {};
\vertex[circle,  minimum size=8pt](6) at  (3.5,1) {};
\vertex[circle,  minimum size=8pt](7) at  (5.5,1) {};
\draw (1)--(2) -- (3) -- (4)--(1);
\draw (1)--(5) -- (2);
\draw (3)--(6) -- (4);
\draw (6) -- (7);
\draw (1,-1) ;
%node{$H^1_6$}; 
\end{scope}
\end{tikzpicture}
\caption{The three known minimal counter-examples with a bridge to $\ell(G)+br(G)\geq |G|$.}

\end{figure}
 
\noindent Since for these three graphs, the bridge is a pendant edge, we venture to propose the following conjecture.

\begin{conjecture}\label{con:amrz}
There is a finite set of graphs ${\cal F}_0$ such that every connected graph $G \notin \mathcal F_0$ either has a pendant edge or satisfies  $\ell(G)+br(G)\geq |G|$.
\end{conjecture}
 
So far, we know that if such a family ${\cal F}_0$ exists, it contains the  list of graphs in Figure~\ref{fig} and Figure ~\ref{fig2}. An interesting variation of the conjecture can be stated as follows, denoting by $ul(G)$ the number of pairs of vertices in $G$ that induce a universal line.

\begin{conjecture}\label{con:amrz2}
For every connected graph $G$, $\ell(G)+ul(G)\geq | G |$.
\end{conjecture}

Although less general (a bridge always induces a universal line but not all universal lines are induced by bridges), this conjecture has the merit of being true for all the known graphs in $\mathcal F$. Thus, there is no known counter-examples to Conjecture \ref{con:amrz2} to this day. Moreover, it remains stronger than the original Chen-Chvatal conjecture without ruling out graphs with universal lines as trivial solutions.

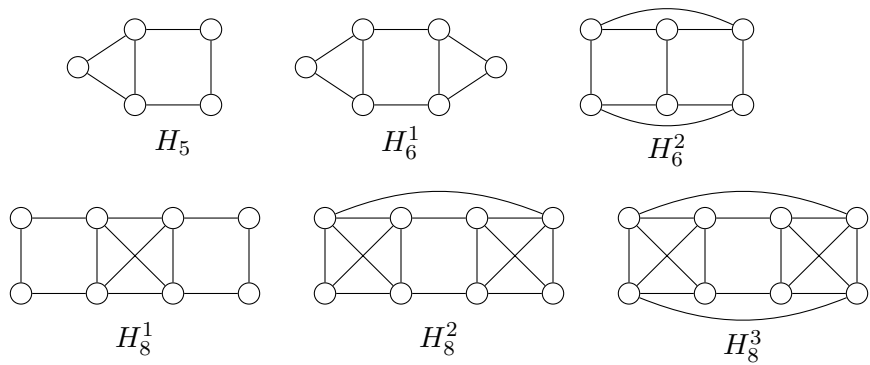
\begin{figure}[h]\label{fig2}
\centering
\begin{tikzpicture}[scale=0.5]
\begin{scope}[xshift=0cm]
\vertex[circle,  minimum size=8pt](1) at  (0,0) {};
\vertex[circle,  minimum size=8pt](2) at  (0,2) {};
\vertex[circle,  minimum size=8pt](3) at  (2,2) {};
\vertex[circle,  minimum size=8pt](4) at  (2,0) {};
\vertex[circle,  minimum size=8pt](5) at  (-1.5,1) {};
\draw (1)--(2) -- (3) -- (4)--(1);
\draw (1)--(5) -- (2);
\draw (1,-1) node{$H_5$}; 
\end{scope}

\begin{scope}[xshift=6cm]
\vertex[circle,  minimum size=8pt](1) at  (0,0) {};
\vertex[circle,  minimum size=8pt](2) at  (0,2) {};
\vertex[circle,  minimum size=8pt](3) at  (2,2) {};
\vertex[circle,  minimum size=8pt](4) at  (2,0) {};
\vertex[circle,  minimum size=8pt](5) at  (-1.5,1) {};
\vertex[circle,  minimum size=8pt](6) at  (3.5,1) {};
\draw (1)--(2) -- (3) -- (4)--(1);
\draw (1)--(5) -- (2);
\draw (3)--(6) -- (4);
\draw (1,-1) node{$H^1_6$}; 
\end{scope}

\begin{scope}[xshift=12cm, yshift=0cm]
\vertex[circle,  minimum size=8pt](1) at  (0,0) {};
\vertex[circle,  minimum size=8pt](2) at  (2,0) {};
\vertex[circle,  minimum size=8pt](3) at  (4,0) {};
\vertex[circle,  minimum size=8pt](4) at  (0,2) {};
\vertex[circle,  minimum size=8pt](5) at  (2,2) {};
\vertex[circle,  minimum size=8pt](6) at  (4,2) {};
\draw (4)--(5)--(6);
\draw (1)--(2)--(3);
\draw (1)--(4);
\draw (2)--(5);
\draw (3)--(6);
\draw(1)  to[bend right=25]   (3);  
\draw (4)  to[bend left=25]   (6); 
\draw (2,-1.2) node{$H^2_6$};  
\end{scope}

%%%%%%%%%%%%%%%%%ù
\begin{scope}[xshift=-3cm, yshift=-5cm]
\vertex[circle,  minimum size=8pt](1) at  (0,0) {};
\vertex[circle,  minimum size=8pt](2) at  (2,0) {};
\vertex[circle,  minimum size=8pt](3) at  (4,0) {};
\vertex[circle,  minimum size=8pt](4) at  (6,0) {};
\vertex[circle,  minimum size=8pt](5) at  (0,2) {};
\vertex[circle,  minimum size=8pt](6) at  (2,2) {};
\vertex[circle,  minimum size=8pt](7) at  (4,2) {};
\vertex[circle,  minimum size=8pt](8) at  (6,2) {};
\draw (1)--(2)--(3) --(4) --(8) --(7) --(6) --(5) --(1);
\draw (2)--(7)--(3)--(6)--(2) ;
\draw (3,-1.2) node{$ H_8^1$};  
\end{scope}

\begin{scope}[xshift=5cm, yshift=-5cm]
\vertex[circle,  minimum size=8pt](1) at  (0,0) {};
\vertex[circle,  minimum size=8pt](2) at  (2,0) {};
\vertex[circle,  minimum size=8pt](3) at  (4,0) {};
\vertex[circle,  minimum size=8pt](4) at  (6,0) {};
\vertex[circle,  minimum size=8pt](5) at  (0,2) {};
\vertex[circle,  minimum size=8pt](6) at  (2,2) {};
\vertex[circle,  minimum size=8pt](7) at  (4,2) {};
\vertex[circle,  minimum size=8pt](8) at  (6,2) {};
\draw (1)--(2)--(3) --(4) --(8) --(7) --(6) --(5) --(1);
\draw (1)--(6)--(2)--(5);
\draw (4)--(7)--(3)--(8);
\draw (5)  to[bend left=22]   (8);
\draw (3,-1.2) node{$ H_8^2$};  
\end{scope}

\begin{scope}[xshift=13cm, yshift=-5cm]
\vertex[circle,  minimum size=8pt](1) at  (0,0) {};
\vertex[circle,  minimum size=8pt](2) at  (2,0) {};
\vertex[circle,  minimum size=8pt](3) at  (4,0) {};
\vertex[circle,  minimum size=8pt](4) at  (6,0) {};
\vertex[circle,  minimum size=8pt](5) at  (0,2) {};
\vertex[circle,  minimum size=8pt](6) at  (2,2) {};
\vertex[circle,  minimum size=8pt](7) at  (4,2) {};
\vertex[circle,  minimum size=8pt](8) at  (6,2) {};
\draw (1)--(2)--(3) --(4) --(8) --(7) --(6) --(5) --(1);
\draw (1)--(6)--(2)--(5);
\draw (4)--(7)--(3)--(8);
\draw (5)  to[bend left=22]   (8);
\draw (1)  to[bend right=22]   (4);
\draw (3,-1.4) node{$ H_8^3$};  
\end{scope}
\end{tikzpicture}
  \caption{Known graphs in  ${\cal F}_0 \setminus \mc F$.}
  
\end{figure}

\section{Preliminaries}

In this section, we give some results on the number of lines of graphs in $\mc F$ or that are constructed from a graph in $\cal F$ adding a vertex. The proof of the following  Lemma is done by brute force using a computer \footnote{
The details can be found \href{http://www.math.sciences.univ-nantes.fr/~rochet/recherche/Lines_in_F.pdf}{\url{http://www.math.sciences.univ-nantes.fr/~rochet/recherche/Lines\_in\_F.pdf.}}}.
(See Figure 4.)

\begin{lemma}\label{lem:smallcase} $\ell(C_4)=1$, $\ell(K'_6)=4$ and 
for $H \in \mc F \setminus \{ C_4,K'_6 \}$, $\ell(H) = | H | -1$. 
%$\ell(H) \ge |H|+br(H)-1$. 
\end{lemma}

\noindent

\begin{lemma}\label{lem:smallcaseextension}
Let $G\in {\cal C}\setminus \cal F$ be a graph. 

1) If $G$ has a pendant vertex $v$ such that $G-v \in \mathcal{F} \setminus \{ C_4 \}$, then $\ell(G)+br(G)=\ell(G)+1\geq |G|$.

2) If $G$ contains a non-trivial module $M$ and $G-v\in {\cal F}\setminus \{ C_4 \}$, for some
$v\in M$, then $\ell(G)\geq |G|+1$.
\end{lemma}
\begin{myp}{} The proofs of both statements are easy although tedious. In the first case, if $u$ is the neighbor of $v$ in $G$, then $\ol{wv}^G$ defines
different lines, when $w$ varies over the neighbors of $u$ in $G-v$. 
These lines are not in ${\cal L}(G-v)$ if $G-v \in  \mathcal{F} \setminus \{ C_4 \}$ . Since the graphs in $\mathcal{F} $ have no vertex of degree one, we obtain at least two new lines. 

In the second case, for each $G'\in \cal F$ we need to consider  all graphs
$G$ arising from $G'$ by adding a copy $v$ of a vertex $v'$ in $G'$ so as $(v,v')$ is a pair of twins (true or false) in $G$. We do this with the help of a computer program \footnote{
For the sake of completeness, the \texttt{R} code and environment used to check all the cases are available in  \href{http://www.math.sciences.univ-nantes.fr/~rochet/recherche/Code_lines.R}{\url{http://www.math.sciences.univ-nantes.fr/~rochet/recherche/Code\_lines.R}} and in  \href{http://www.math.sciences.univ-nantes.fr/~rochet/recherche/env_lines.RData}{\url{
http://www.math.sciences.univ-nantes.fr/~rochet/recherche/env\_lines.RData.}}
}.

\end{myp}

\begin{figure}\label{fig:smallbadcases}
\begin{tabular}[c]{|c|c|}
  \hspace*{-0.6cm} \includegraphics[width=0.55\textwidth]{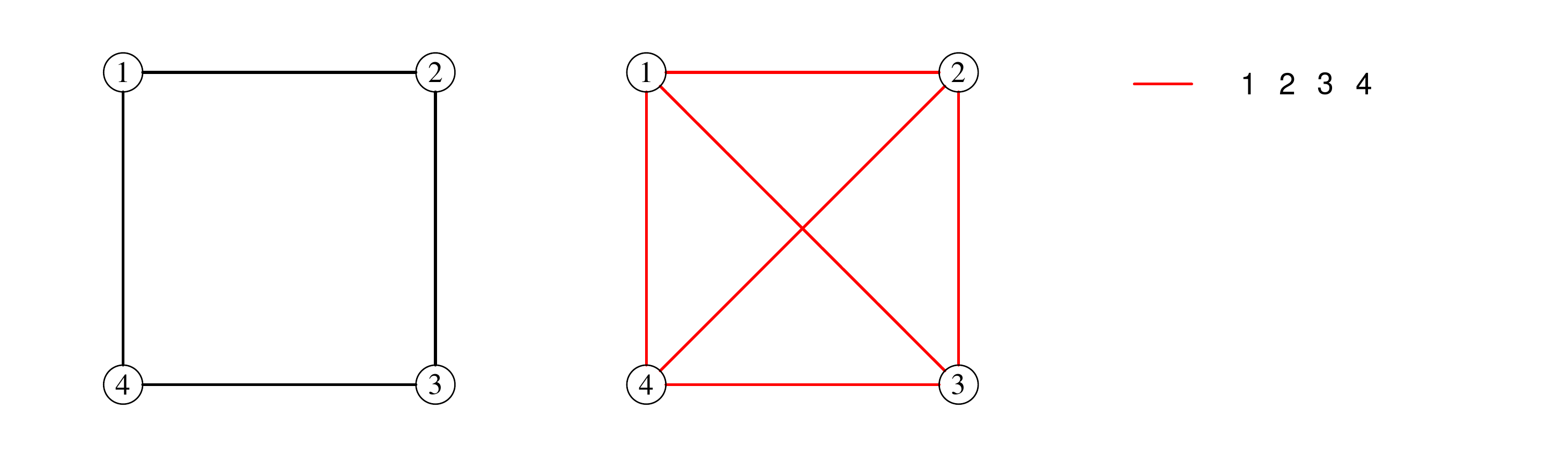} \hspace*{-0.95cm} & 
  \hspace*{-0.6cm} \includegraphics[width=0.55\textwidth]{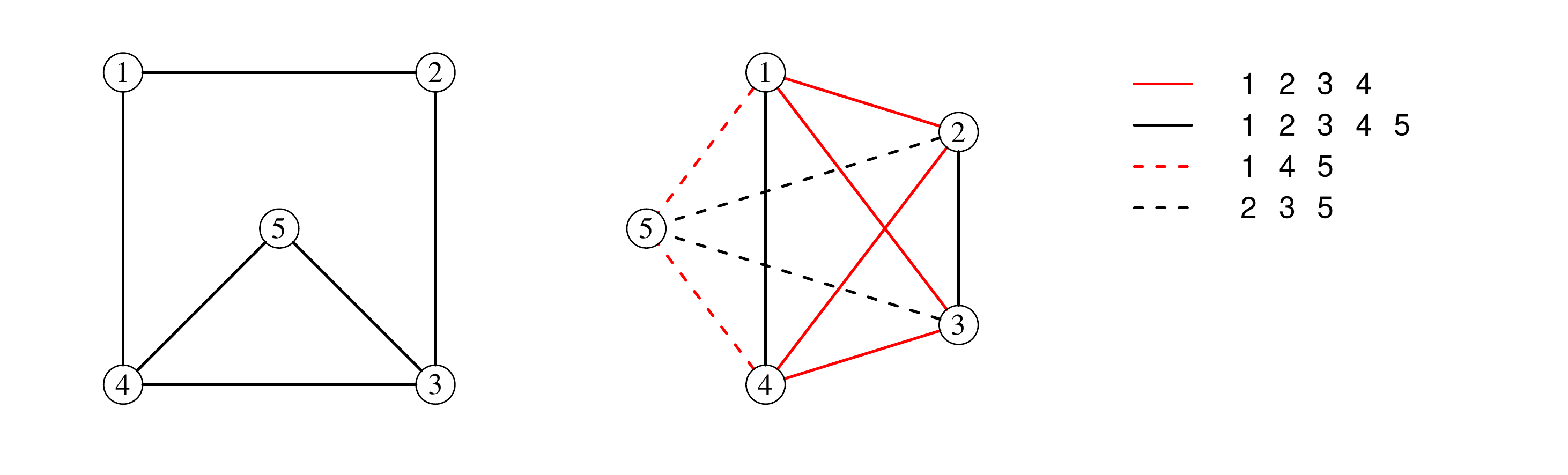} \hspace*{-0.9cm}
 \\ \hline
  \hspace*{-0.6cm} \includegraphics[width=0.55\textwidth]{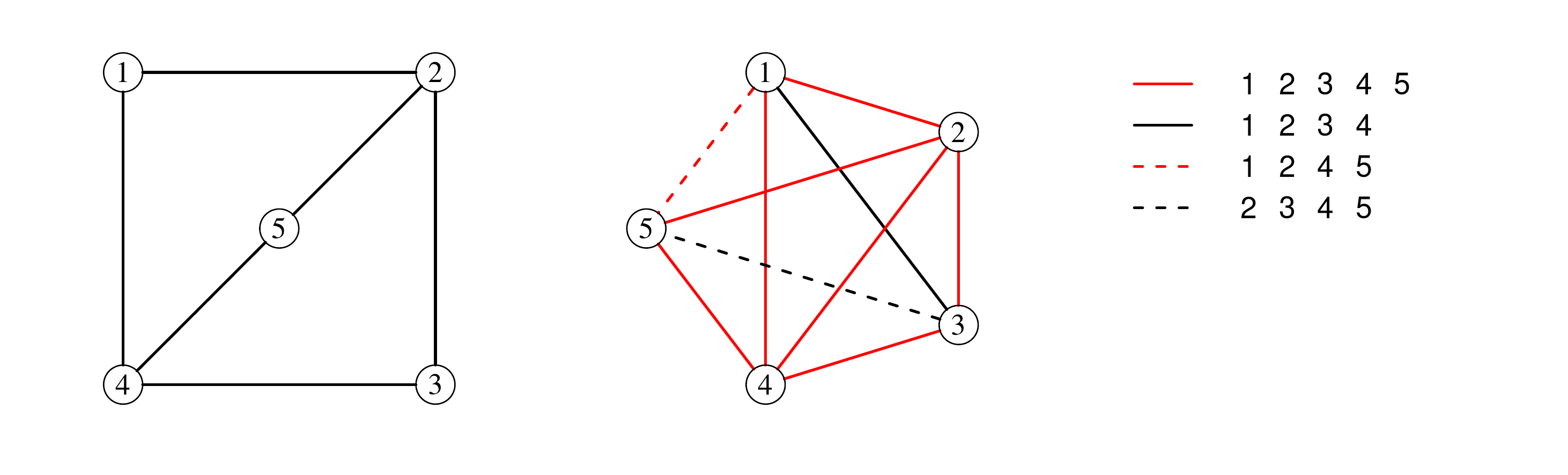} \hspace*{-0.95cm}
 & \hspace*{-0.6cm} \includegraphics[width=0.55\textwidth]{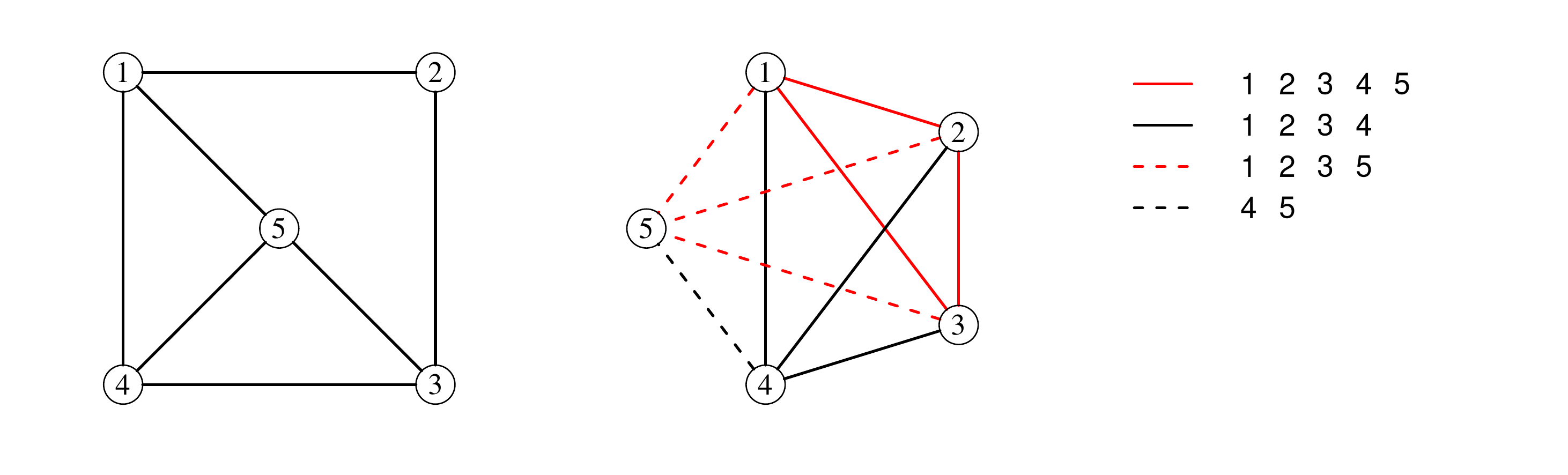} \hspace*{-0.9cm}
 \\ \hline
  \hspace*{-0.6cm} \includegraphics[width=0.55\textwidth]{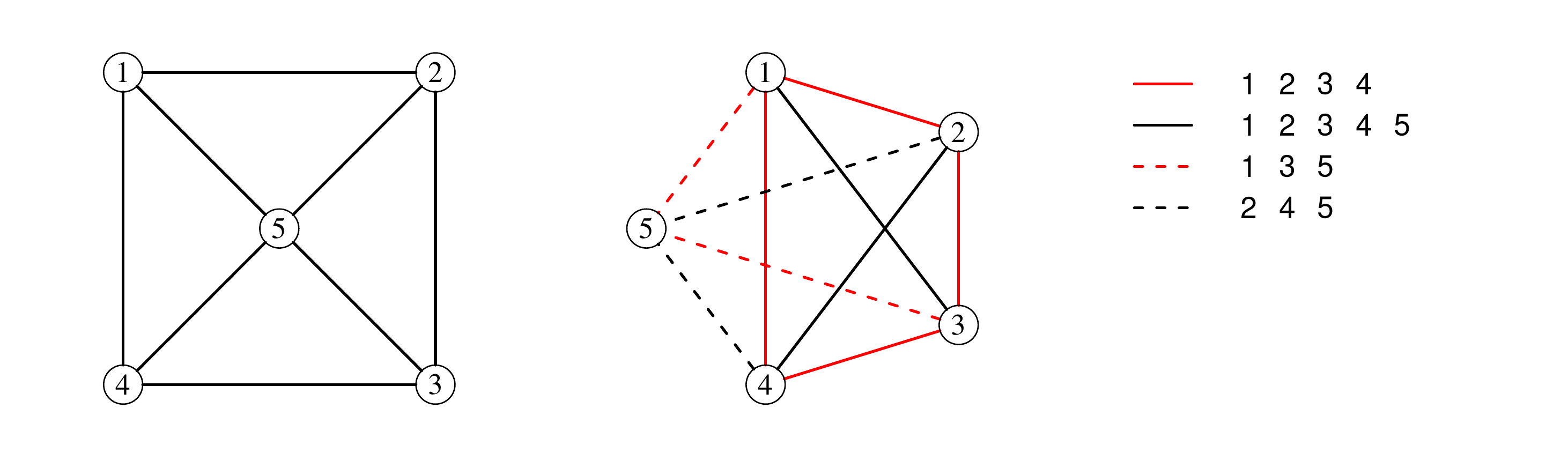} \hspace*{-0.95cm}
& \hspace*{-0.6cm} \includegraphics[width=0.55\textwidth]{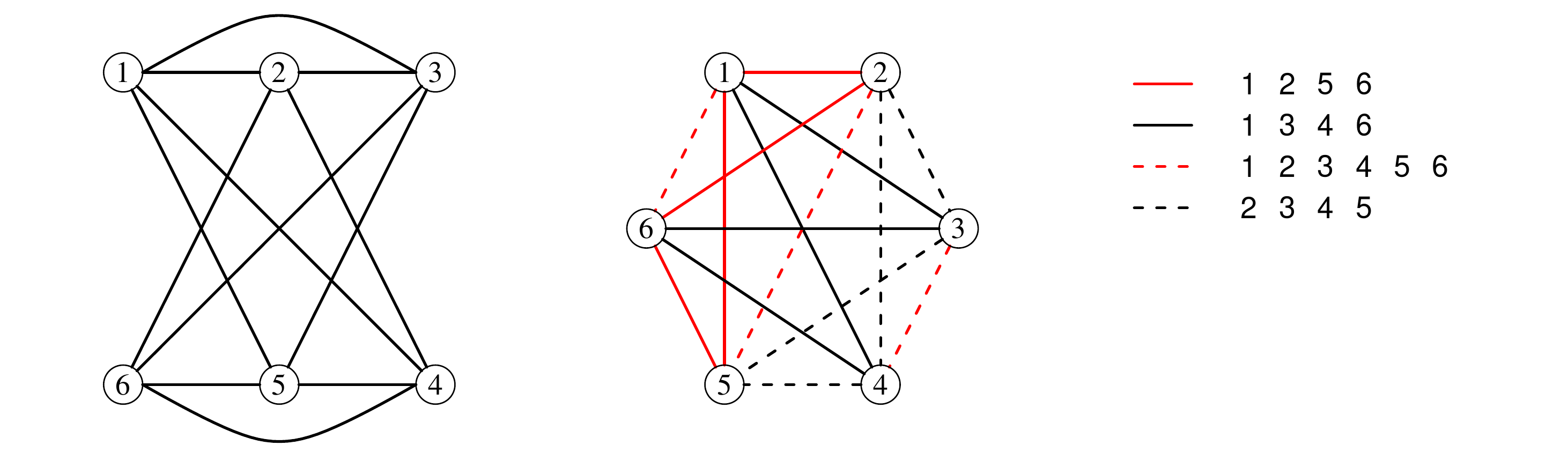} \hspace*{-0.9cm}
\\ \hline
 \multicolumn{2}{|c|}{\includegraphics[width=0.55\textwidth]{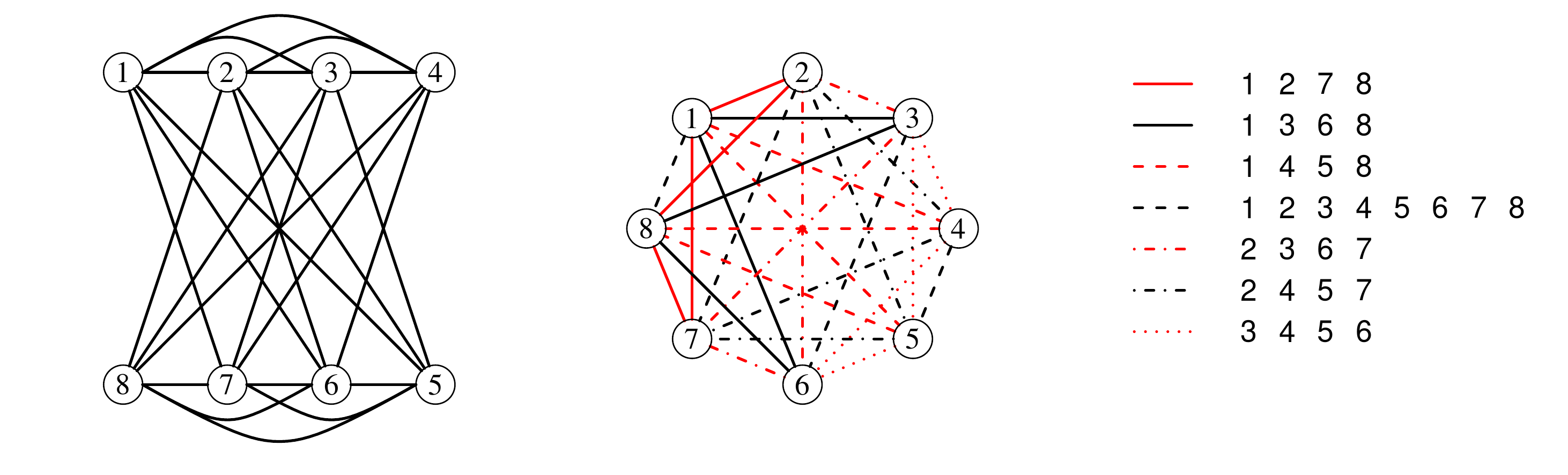} }\\
\hline
\end{tabular}
\caption{Members of ${\cal F}$. For each graph in ${\cal F}$, three drawings appear.
To the left, the graph itself. In the middle, an edge-colored complete
graph where  pairs of vertices defining the same line have the same color.
To the right, the set of vertices in each line (color).}
\end{figure}

\medskip 
\section{Proof of Theorem~\ref{th:main}}

We prove Theorem \ref{th:main} by induction on the number of vertices of $G$. 
Let $G \in \mc C \setminus \mc F$ with $|G|=n$. 
The proof splits in four parts: (1) $G$ has a bridge, (2) $G$ has no bridge and has a  cut-vertex,
(3) $G$ is 2-connected and chordal, (4) $G$ is $2$-connected and has a non trivial module.

%Let $H$ be an induced subgraph of $G$. If $H \in \mc C \setminus \mc F$, then, by the induction hypothesis, $\ell(H) \ge |H|+br(H)$.  During the proof, we shall also use that if $H \in \mc F \setminus C_4$, then $\ell(H) \ge |H|+br(H)-1$.  In several situations this weaker conclusion  suffices to obtain the result. 

%%%%%%%%%%%%%%%%%%%%%%%%%%%%%%%%%%%%%%%%%%%%%%%%%%%%%%%
%%%%%%%%%%%%%%%%%%%%%%%%%%%%%%%%%%%%%%%%%%%%%%%%%%%%%%%

\subsubsection*{Part 1: $G$ has a bridge. }

Let $u_1u_2$ be a bridge of $G$. Let $G_1$ and $G_2$ be the connected components of $G-u_1u_2$ that contains respectively $u_1$ and $u_2$.  
To \emph{contract} an edge $e$ of a graph $G$  is to delete $e$ and then identify its ends. 
Let $G'$ be the graph obtained from $G$ by contracting $u_1u_2$. 
Name $u$  the vertex of $G'$ appeared after the contraction. 
The following claim which easy proof is omitted, say that for any two vertices in 
$G'-\{u\}$ the lines $\ov{xy}^G$ and $\ov{xy}^{G'}$ might differ only in $\{u,u_1,u_2\}$.
\[ \ov{xy}^{G}=
\left\{
\begin{array}{llll}
\ov{xy}^{G'} & \mbox{if $u \notin \ov{xy}^{G'} $}\\
\ov{xy}^{G'}-\{u\}\cup\{u_i\}& \mbox{if $x,y\in V(G_i)$, $[xuy]^{G'}$ and $i \in \{1, 2\}$ }\\
\ov{xy}^{G'}-\{u\}\cup\{u_1,u_2\}& 
\left\{
\begin{array}{ll}
\mbox{if $x,y \in V(G_i)$ for an $i \in \{1,2\}$ and $[xyu]^{G'}$ or $[yxu]^{G'}$ or, }\\
\mbox{if $x \in V(G_i)$ and  $y \in V(G_{3-i})$ for an $i \in \{1,2\}$ }\\
\end{array}\right. \\
\end{array}
\right.
\]
And in the case that $y=u$ we have that:
$$\ov{ux}^{G'}-\{u\}=\ov{u_1x}^G- \{u_1,u_2\}.$$

This implies that $\ell(G) \ge \ell(G')$. 
Moreover it is clear that $br(G)=br(G')+1$. 
If $G' \in \mc C \setminus \mc F$, then by induction we have $\ell(G') + br(G') \ge |G'|=|G|-1$ and thus $\ell(G)+br(G) \ge |G|$ and we are done. 

So we may assume that $G' \in \mc F$. 
Since graphs in $\mc F$ are $2$-connected, it implies that $u_1u_2$ is a pendant edge of $G$. Then the result follows by   Lemma~\ref{lem:smallcaseextension} when $G' \neq C_4$, and it is easily checkable when $G' = C_4$. 
This ends the Part 1.

\subsubsection*{Part 2:  $G$ has no bridge and has a  cut-vertex. }

Let $u$ be a cut-vertex of $G$. 
Let $C_1$ be a connected component of $G-\{u\}$ and let $C_2$ be the union of the other connected components of $G$. 
Set $G_1=G[C_1 \cup \{u\}]$, and $G_2=G[C_2 \cup \{u\}]$. 
Observe that, since $G$ is bridgeless, $G_1$ and $G_2$ are also bridgeless. 

Claim (1) below, whose easy proof is omitted, implies that, for $i=1,2$,  a line induced by two vertices in  $V(G_i)$ is either disjoint from $V(C_{3-i})$ or contains $V(C_{3-i})$. In particular, it  implies that a line induced by two vertices  in $V(G_1)$ is distinct from a line induced by two vertices in $V(G_2)$, except in the perverse case where this line is universal. 

\begin{claim}\label{sub:cut-vertex} 
 For $i=1,2$ and for all $x,y \in V(G_i)$ we have:
 \begin{itemize}
\item   if $[xyu]$ or $[yxu]$,  then $\ov{xy}^G=\ov{xy}^{G_i} \cup V(C_{3-i})$, 
  \item otherwise $\ov{xy}^{G}=\ov{xy}^{G_i}$ and in particular $\ov{xy}^{G} \cap V(C_{3-i})=\emptyset$. 
  \end{itemize}
\end{claim}

We next prove the following lower bound for $\ell(G)$.

\begin{claim}\label{lem:cut-vertex} 
$\ell(G)\geq \ell(G_1)+\ell(G_2)-1+|N_{G_1}(u)||N_{G_2}(u)|$.
\end{claim}

\begin{proofclaim}
For $i=1,2$, let ${\cal L}_i=\{\ov{ab}^G: a,b \in V(G_i)\}$. 
By~(\ref{sub:cut-vertex}) $|{\cal L}_i|=\ell(G_i)$, and  the only possible line in ${\cal L}_1 \cap {\cal L}_2$ is the universal line. Hence $|\mathcal L_1 \cup \mathcal L_2| \ge \ell(G_1)+\ell(G_2)-1$.  
Moreover, for all lines $l$ in $\mc L_1 \cup \mc L_2$, 
%$l\cap V(G_i)\in \{\emptyset,V(G_i)\}$, for each $i=1,2$. c'est faux, une ligne dans L_1 par exemple peut ne contenir qu'une partie de L_1
$l$ contains either $V(C_1)$ or $V(C_2)$. 
%For each $l\in {\cal L}(G_1)$ we have that $l^G\in\{l,l\cup V(G_2)\}$ and for each $l\in {\cal L}(G_2)$ we have that $l^G\in\{l,l\cup V(G_1)\}$. Therefore, at most one line belongs to the intersection ${\cal L}(G_1)^G\cap {\cal L}(G_2)^G$. Therefore, there are at least $\ell(G_1)+\ell(G_2)-1$ lines  in ${\cal L}(G_1)^G\cup {\cal L}(G_2)^G$.

For  $i=1,2$, let $u_i$ be a neighbor of $u$ in $G_i$. 
We have that $\ol{u_1u_2}^G\cap N(u)=\{u_1,u_2\}$ and since $u$ has  at least two neighbors in both  $G_1$ and $G_2$ (because $G$ is bridgeless), then $\ov{u_1u_2}^G\cap V(G_i)\notin \{\emptyset, V(G_i)\}$ for each $i=1,2$ and  thus it is distinct  from all lines in $\mathcal L_1 \cup \mathcal L_2$. 
Moreover, for every $u_i,v_i$ neighbors of $u$ in $G_i$, for each $i=1,2$, if $\{u_1,u_2\}\neq \{v_1,v_2\}$, then $\ol{u_1u_2}^G\neq \ol{v_1v_2}^G$.
Therefore, there are at least $ |N_{G_1}(u)||N_{G_2}(u)|$ lines in ${\cal L}(G)\setminus (\mathcal L_1 \cup \mathcal L_2)$.
\end{proofclaim}

Since $|N_{G_1}(u)||N_{G_2}(u)|\ge 4$ we get $\ell(G) \ge \ell(G_1) + \ell(G_2)+3$.
If  $\ell(G_1) + \ell(G_2)\geq |G_1|+|G_2|-4$, then 
$\ell(G) \ge |G_1|+|G_2|-1=|G|$. If $\ell(G_1)+\ell(G_2)\leq |G_1|+|G_2|-3$, 
by the induction hypothesis, we conclude that both $G_1$ and $G_2$ are in $\cal F$.
Moreover, from Lemma \ref{lem:smallcase} we get that $G_1=G_2=C_4$
or $\{G_1,G_2\}=\{C_4,K_6'\}$. We have verified that in the first 
case we have 11 lines, while in the second, we have 20 lines (see Figure 5). %\ref{fig:smallcutvertex}).
This ends the Part 2. 

\begin{figure}\label{fig:smallcutvertex}
\begin{center}
\begin{tabular}{|c|}
\hline \\
 \includegraphics[width=0.45\textheight]{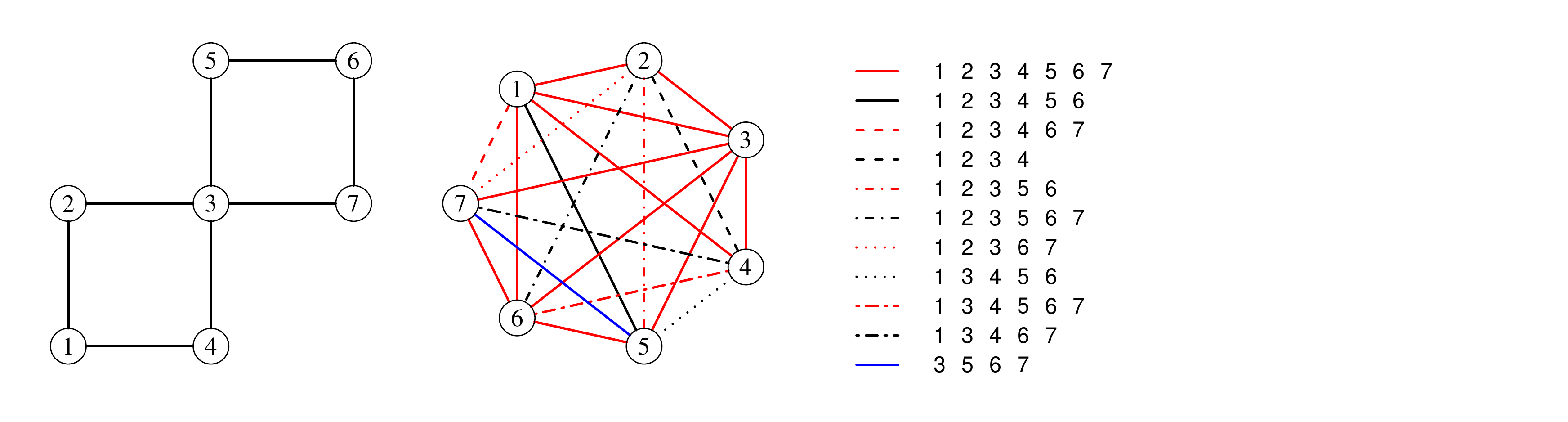}\\ \hline
 \multicolumn{1}{c}{}\\
\hline
 \includegraphics[width=0.45\textheight]{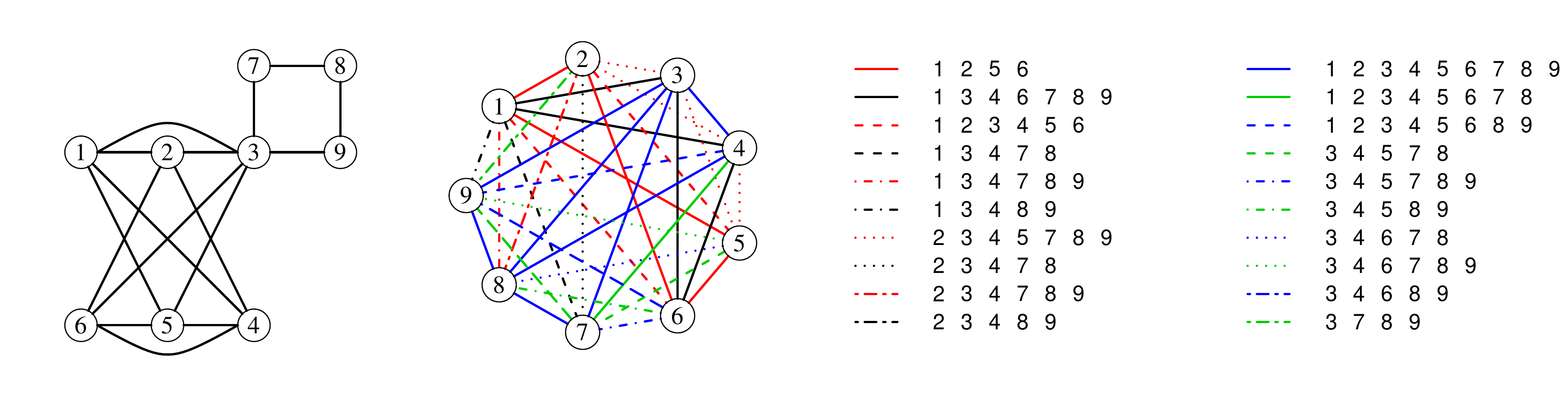}\\
 \hline
 \end{tabular}
 \end{center}
 \caption{Lines of the graphs obtained by gluing, respectively, a cycle of length four and a $K_6'$ to a $C_4$. To the right  we represent the set of lines as an edge coloring of the complete graph, together with the elements in each set.
 }
\end{figure}

%%%%%%%%%%%%%%%%%%%%%%%%%%%%%%%%%%%%%%%%%%%%%%%%ù
%%%%%%%%%%%%%%%%%%%%%%%%%%%%%%%%%%%%%%%%%%%%%%%%%%%%%%%
\subsubsection*{Part 3: $G$ is 2-connected and chordal.}

In \cite{BBCCCCFZ} it was proved that Conjecture \ref{conjC} holds for 
chordal graphs. The proof of this part is the same as their proof. 
We first need Lemma 1 of  \cite{BBCCCCFZ}:

\begin{lemma}\label{lem:chordalchvatal}
Let $G$ be a chordal graph and let $s,x,y$ in $V(G)$ such that $[sxy]$. 
If $\ov{sx}=\ov{sy}$, then $x$ is a cut-vertex of $G$. 
\end{lemma}

A vertex of a graph is called \emph{simplicial} if its neighbors are pairwise adjacent. 
By a classic result of Dirac~\cite{dirac}, a chordal graph has at least two simplicial vertices. 
Let $s$ be a simplicial vertex of $G$. 
Since $s$ is simplicial for any pair of vertices $x,y \in V(G)\setminus \{s\}$, $[xsy]$ does not hold. 
Hence, if $\ov{sx}=\ov{sy}$, we must have $[sxy]$ or $[syx]$ and thus, by Lemma~\ref{lem:chordalchvatal}, $x$ or $y$ is a cut vertex, a contradiction. 

Hence, the set $\{\ol{su}:u\in V(G) \setminus \{s\}\}$ has
$n-1$ distinct lines. Observe that all these lines contain $s$.  
Now, since $G$ is $2$-connected, $s$ has at least two neighbors $a,b$ and $s \notin \ov{ab}$.

\subsubsection*{Part 4:  $G$ is $2$-connected, non-chordal and has a non trivial module}

\emph{We first consider the case when $G$ has a non-trivial non-dominating module.}

Let $M=\{v_1, \dots v_s\}$ be a non-trivial non-dominating module  of $G$ with neighborhood $N(M)$ of minimal size.
% and let $N(M)$ be the vertices of $G-M$ having a neighbor in $M$. Then every pair of vertices $u\in M$ and $u'\in N(M)$ are adjacent.  Let $M$ be a non-trivial module of $G$ with $|N(M)|$ as small as possible.
%Among all such modules, choose the one that minimized $|N(M)|$.
Set  $G':=G-\{v_1\}$.   
If $G'\in {\cal F}\setminus \{ C_4 \}$, we are done by  Lemma \ref{lem:smallcaseextension}. 
If $G'=C_4$, then $G$ is either $K_{2,3}$ or $W_4'$, a contradiction with $G\notin {\cal F}$.
So we may assume that  $G'\notin \cal F$ and from the induction hypothesis we get
$\ell(G')+br(G')\geq |G'|=|G|-1$.

Set $\mc L'=\{\ov{xy}^G: x,y \in V(G')\}$. 
Since $G'$ is an isometric subgraph of $G$ (i.e. for all $x,y \in V(G')$, $d_{G'}(x,y)=d_G(x,y)$), we have,  for all $a,b \in   V(G')$, $\ov{ab}^G=\ov{ab}^{G'}$ or $\ov{ab}^{G'} \cup \{v_1\}$. 
Hence 
\begin{equation}\label{eq4:1}
|\mc L'| = \ell(G') \ge |G|-1-br(G').
\end{equation}
Moreover, each line in  $\mc L'$ that contains $v_1$ must contain at least one other vertex of $M$. In effect, let $\ov{ab} \in \mc L'$ such that $v_1 \in \ov{ab}$. 
If $a,b \notin M$, then $\ov{ab}$  contains $M$ and we are done, otherwise either $a$ or $b$ are  in $M$ and thus $\ov{ab}$ contains at least two vertices of $M$.

Let  $t\in G-(M\cup N(M))$. It is clear that $v_1$ is the unique vertex in $M$ which belongs to the line $\ol{v_1t}^G$. 
Hence, $\ol{v_1t}^G\notin \mc L'$ and thus, if $br(G')=0$, we are done by (\ref{eq4:1}).

So we may assume that $G'$ has at least one bridge. 
Let $ab$ be a bridge of $G'$, and let $G_a$, $G_b$ be the connected components of $G'-ab$ that contains respectively $a$ and $b$.
We are going to prove that one of $G_a$, $G_b$ is reduced to one vertex of degree exactly $2$ and that this vertex is in $N(M)$ (so it  also implies that $|M|=2$). 
Since a vertex in $G_a$ has at most one common neighbor with a vertex in $G_b$ and $|N(M)| \ge 2$, because $G$ is $2$-connected and thus two vertices in $M-\{v_1\}$ have  at least  two common neighbors in $G'$, $M-\{v_1\}$ cannot intersect both $G_a$ and $G_b$.   
So we may assume without loss of generality that $M-\{v_1\} \subseteq V(G_a)$. 
Since $ab$ is not a bridge of $G$, $v_1$ must have a neighbor in both $G_a$ and $G_b$. 
Hence the only neighbor of $v_1$ in $G_b$ is $b$, $v_2=a$ and $M=\{v_1,v_2\}$. 
Moreover, since $G$ has no cut-vertex, $G_b=\{b\}$. 
Finally, by minimality of $N(M)$, $v_2b$ is the unique bridge of $G'$. 
Indeed, if $G'$ has another bridge, then there exists a vertex $b' \neq b$ such that $v_2b'$ is a bridge of $G'$ and $N_G(b')=\{v_1,v_2\}$. Hence $\{b,b'\}$ is a non trivial non-dominating module of $G$ and $|N(\{b,b'\})|<|N(\{v_1,v_2\})|$, a contradiction.

Consider now the line $\ov{v_1v_2}$. 
We claim that $\ov{v_1v_2} \notin \mc L' \cup \{\ov{v_1t}\}$ which gives the result by (\ref{eq4:1}). 
If $v_1v_2$ is an edge of $G$, then $\ov{v_1v_2}=\{v_1,v_2\}$ and the result holds. 
Hence we may assume that $v_1v_2$ is not an edge and thus $\ov{v_1v_2}=M \cup N(M)$. 
So $\ov{v_1v_2} \neq \ov{v_1t}$ and we may assume for contradiction that $\ov{v_1v_2} \in \mc L'$ which implies there exists $x,y \in N(M) \cup M - \{v_1\}$ such that $\ov{xy}=M \cup N(M)$. 
If  $\{x,y\} \cap \{b,v_2\} \neq \emptyset$, then $\ov{xy}$ must contain some vertices of $V(G)-(M \cup N(M))$, so we may assume that $\{x,y\} \subseteq N(M)$, but then $b \notin \ov{xy}$.

\medskip

\noindent\emph{We now consider the case where all the  non-trivial modules of $G$ are dominating.}
\\
In this case, $G$ has diameter $2$. 
It was  proven in \cite{Chvatal2} that for every graph $G$ of diameter $2$, $G$ either has an universal line or it has at least $|V(G)|$ distinct lines. Since what we want to prove is stronger, we cannot use this result. 
We will need the following lemma that was already proved in \cite{ChCh11}.
\begin{lemma}\label{lem:diamtwo}
Let $G$ be a graph  of diameter two  and let  $x,a,b$ be three vertices of $G$ such that $\ov{xa}=\ov{xb}$. 
 Then either $(a,b)$ is a pair of  false twins and $d(x,a)=d(x,b)=1$, or $d(x,a)\neq d(x,b)$.
\end{lemma}
\begin{proof}
Assume that $d(x,a)= d(x,b)$. If $d(x,a)=2$, then $a \notin \ov{xb}$, a contradiction, so $d(x,a)=1$. If $a$ and $b$ are adjacent, then again $a \notin \ov{xb}$, so $a$ and $b$ are not adjacent. Assume now that there exists a vertex $c$ adjacent to $a$ but not to $b$, i.e. $d(c,a)=1$ and $d(c,b)=2$. If $d(c,a)=1$, then $c \in \ov{xb}$ and $c \notin \ov{xa}$, and if  $d(c,x)=2$, then $c \notin \ov{xb}$ and $c \in \ov{xa}$, a contradiction in both cases. So $(a,b)$ is a pair of  false twin. 
\end{proof}

Notice that for any non-trivial dominating module $M$, the set $N(M)$ is a module as well
and $M\cup N(M)=V(G)$. 
Moreover, for each $u\in M$ and each $v\in N(M)$ the line $\ov{uv}$ is given by 
$$\ol{uv}=   (M-N(u)) \cup (N(M)-N(v)).$$

We assume first that $G$ does not  contain pairs of false twins.  
Let $M$ be a module of $G$. 
For $u, u'\in M$ and $v, v'\in N(M)$ with $\{u,v\}\neq \{u',v'\}$ we have that 
$\ol{uv}\neq \ol{u'v'}$. Hence,  $\ell(G)\geq |M||N(M)|$.
Since $|M|\geq 2$, then the equality $|M||N(M)|=|M|+|N(M)|+(|M|-1)(|N(M)-1)-1$
implies that $\ell(G)\geq |G|$ when $N(M)$ is not a singleton.
If $N(M)=\{x\}$, then all the lines $\ol{xv}$ are distinct, when 
$v$ varies over $M$. This gives us $|V(G)|-1$ distinct lines, all containing $x$. 
Since $M$ has no pair of false twins, it contains at least one edge $ab$, and $\ov{ab}$ is a new line since $x \notin \ov{ab}$.

Hence, we can assume that every non-trivial module $M$ is a dominating set and the graph $G$ contains pairs of false twins.

Let $(u_1,v_1), (u_2,v_2), \dots, (u_t,v_t)$ be the pairs of false twins of $G$ and set $T=\{u_1,v_1, \dots, u_t,v_t\}$. 
Since $\{u_i,v_i\}$ is a non-trivial module for $i=1, \dots, t$, it must be  a dominating module and thus $N(u_i)=N(v_i)=V(G)\setminus \{u_i,v_i\}$. This implies that all vertices in $\{u_1,v_1, \dots, u_t,v_t\}$ are pairwise distinct, i.e. $|T|=2t$, and that $T$ induces a complete graph minus a perfect matching. 

Set $U= \{u_1, \dots, u_t\}$ and $\mathcal L_U=\{\ov{u_iu_j}: 1 \le i \neq j \le t\} \cup \{ \ov{u_1v_1} \}$. 
For $1 \le i\neq j \le t$, we have $\ov{u_iu_j}=\{u_i,v_i,u_j,v_j\}$  and $\ov{u_1v_1}=V(G)$. 
So $|\mathcal L_U|={t \choose 2}+1$.

Set $R=V(G)\setminus T$. We  split the rest of the proof in three cases. 
\medskip 

\noindent{\bf Case 4.1:} $R$ is empty.
\\
In this case $|V(G)|=2t$. 
If $t\in \{2,3,4\}$, then $G \in \{C_4, K'_6,K'_8\}$, which is a contradiction. 
If $t \ge 5$, then $\ell(G)\geq |\mathcal L_U| \geq {t \choose 2}\ge 2t$ and we are done.
\medskip 

\noindent{\bf Case 4.2:} $R$ is a clique.
\\
Set $|R|=k\ge 1$. 
Set $\mathcal L_R=\{\ov{xy}:x,y \in R\}$.
Notice that each pair of vertices $x,y \in R$ are true twins, resulting in $\ov{xy}=\{x,y\}$ (they are uniquely determined).
Now, for each $x \in R$, set $\mathcal L_{xU}=\{\ov{xu_i}: i=1, \dots, t\}$. Observe that $\ov{xu_i}=\{x,u_i,v_i\}$. It follows that 
lines in $\cup_{x \in R} \mathcal L_{xU}$ are all pairwise distinct and disjoint of $\mathcal L_U$. 
Moreover, lines in $\cup_{x \in R} \mathcal L_{xU}$ are not universal (except when $|R|=1$ and $T=\{u_1,v_1\}$, but then the graph is not 2-connected). 
Hence, all lines in $\mathcal L_U$, $\mathcal L_R$ and $\cup_{x \in R} \mathcal L_{xU}$ are pairwise distinct. 
So if $|R|$ and $t$ are greater than 2 we have that:
\begin{equation}\label{2.2}
\ell(G) \geq {t \choose 2} + {|R| \choose 2} +t|R| +1 \ge 2t + |R|=|V(G)|
\end{equation}

If $|R|=1$ and $t \geq 2$, $\ell(G) \geq {t \choose 2} + t +1$. If $t \geq 3$ this quantity is greater than $|V(G)|$. If $t=2$ then $G = W_4$ which is a contradiction because $W_4 \in \mathcal F$.
If $|R|=1$ and  $t=1$, then $G$ is not $2$-connected. Hence, $R$ is not a clique. 

\medskip 

\noindent{\bf Case 4.3:} $R$ is non empty and is not  a clique.
\\
There exists $x,y \in R$ such that $xy$ is not an edge. 
Since $(x,y)$ is  not a pair of false twin,  we may assume that there exists $z \in R \setminus \{x,y\}$ such that $z$ is adjacent to $y$ but not to $x$. 
Set $\mathcal L_{x}=\{\ov{xa}: a \in U \text{ or } a \in R \setminus \{x\}\}$. 

Suppose $\mathcal L_{x}$ contains an universal line $\ov{xa}$. If $a \in R$, then $d(x,a) = 2$ and all the other vertices in $R$ are at distance 1, but this would imply that $(x,a)$ is a pair of false twins. Hence, $a = u_i$ for some $i \in \{1,2,\ldots, t\}$. Notice that $\ov{xu_i} = \{v_i\} \cup (R \setminus N(x))$. If it is universal, then $i=1, t=1$ and $N(x) = T$. But then, $R \setminus \{x\}$ is a non-trivial non-dominanting module which is a contradiction. Hence, we can assume that 
$\mathcal L_{x}$ does not contain an universal line. 
%which implies that $t \geq 2$ or $N(x) \neq T$.

Recall that all lines in $\mathcal L_U$ are included in $T$ except for the universal line. 
Since no line in $\mathcal L_x$ is universal, then $\mathcal L_x \cap \mathcal L_U=\emptyset$.  
Moreover, since $x \notin \ov{yz}$, $\ov{yz} \notin \mathcal L_x$ and since $\ov{yz} \cap T=\emptyset$, $\ov{yz} \notin \mathcal L_U$. 
Hence, we have that $\ell(G) \ge {t \choose 2} + 2 + |\mathcal L_x|$ if $t \geq 2$ or $\ell(G) \ge   2 + |\mathcal L_x|$ if $t=1$.

In both cases if $|\mathcal L_x| \ge |R|+t-1$ then $\ell(G) \ge 2t+|R| = |G|$. Thus it is enough to prove that  for all $a,b \in U \cup R \setminus \{x\}$, we have $\ov{xa}\neq \ov{xb}$. 

Let $a,b \in  U \cup R \setminus \{x\}$ and let us prove that $\ov{xa} \neq \ov{xb}$. 
Since $a,b \in  U \cup R \setminus \{x\}$, $(a,b)$ is not a pair of false twins, and thus, by Lemma~\ref{lem:diamtwo}, we may assume that  $d(x,a) \neq d(x,b)$. Without loss of generality, $d(x,a)=1$ and $d(x,b)=2$ which implies in particular that $b \in R \setminus \{x\}f$. 
If $a \in R \setminus \{x\}$, then $T \subseteq \ov{xb}$ and $T \cap \ov{xa}=\emptyset$. 
So we may assume that $a \in U$. 
One of the vertices $y,z$ is distinct from $b$, say $y \neq b$. 
We have $[xay]$, so $y \in \ov{xa}$. 
But $d(x,y)=d(x,b)=2$  which implies that $y \notin \ov{xb}$. 

%If $t=1$ and $x$ has a neighbor in $R$ the previous argument works as well because there is not universal line in $\mathcal L_x$, then $\ell(G) \geq |\mathcal L_x| + 2 = |R| +2 = |G|$. Now we assume that $N(x) = T = \{ u_1,v_1\}$ which implies $\ov{xu} = V(G)$. If $R=\{ x, y, z \}$ then $G = W'_4$ which is a contradiction. Hence, $|R| \geq 4$.
%If there exist a vertex $w \neq x$ with $N(w) = T$, the pair  $(x,w)$ would be a pair of false twins which is a contradiction.  Hence, all the vertices in $R-x$ have at least one neighbor in $R$.
%
%In this case $|\mathcal L_x| = |R|$ but $\mathcal L_x$ contains the universal line $\ov{xu}$. We also have that $\ov{yz}$ is not in $\mathcal L_x$ since does not contain $x$. Thus, we need at least one new line.  Let $w \in R\setminus \{x,y,z\}$. If $w$ is not a neighbor of $y$ or $z$, then $\ov{yw}$ neither contain $x$ nor $z$. Then is a new line and $G$  has at least $|R| + 2$ lines.   
%
%Without loss of generality we suppose that $wz$ is an edge of $G$.  If $wy$ is an edge of $G$, then $\ov{yw}$ does not contain $z$ which implies it is a new line. If $wy$ is not an edge of $G$, then $\ov{yz}$ does not contain $w$ which implies that $\ov{yw}$  is a new line. 
%
Thus, $\ell(G) \geq |R| +2 = |G|$ which proves the Theorem.

 %%%%%%%%%%%%%%%%%%%%%%%%%%%%%%%%%%%%%%%%%%%%%%
 %%%%%%%%%%%%%%%%%%%%%%%%%%%%%%%%%%%%%%%%%%%%%%
 %%%%%%%%%%%%%%%%%%%%%%%%%%%%%%%%%%%%%%%%%%%%%%
 %%%%%%%%%%%%%%%%%%%%%%%%%%%%%%%%%%%%%%%%%%%%%%
  %%%%%%%%%%%%%%%%%%%%%%%%%%%%%%%%%%%%%%%%%%%%%%
 %%%%%%%%%%%%%%%%%%%%%%%%%%%%%%%%%%%%%%%%%%%%%%
  %%%%%%%%%%%%%%%%%%%%%%%%%%%%%%%%%%%%%%%%%%%%%%
 %%%%%%%%%%%%%%%%%%%%%%%%%%%%%%%%%%%%%%%%%%%%%%

\end{document}